\documentclass[11pt]{amsart}

\usepackage{amsmath,amsthm,amssymb,graphicx,listings,color,enumerate,url}
\DeclareFontShape{OT1}{cmtt}{bx}{n}{<5><6><7><8><9><10><10.95><12><14.4><17.28><20.74><24.88>cmttb10}{}

\newtheorem{thm}{Theorem}
\newtheorem{remark}[thm]{Remark}
\newtheorem{lem}[thm]{Lemma}
\newtheorem{cor}[thm]{Corollary}

\def\X{\mathfrak X}

\def\th{\theta}

\def\ka{\kappa}

\def\ph{\varphi}

\def\i{^{-1}}

\def\p{\partial}
\let\on=\operatorname

\def\R{{\mathbb R}}
\def\Tr{\on{Tr}}
\def\vol{{\on{vol}}}

\def\Imm{{\on{Imm}}}

\def\Diff{{\on{Diff}}}
\def\g{\overline{g}}

\def\Hor{{\on{Hor}}}
\def\Nor{{\on{Nor}}}
\def\Ver{{\on{Ver}}}
\def\Tan{{\on{Tan}}}
\def\Adj{{\on{Adj}}}

\begin{document}

\title[Metrics where horizontality equals normality]{Metrics on  spaces of immersions where horizontality equals normality}
\author{Martin Bauer and Philipp Harms}
\address{
Martin Bauer
Fakult\"at f\"ur Mathematik, Universit\"at Wien
\newline\indent
Philipp Harms: Department of Mathematics, ETH Z\"urich}
\email{Bauer.Martin@univie.ac.at}
\email{philipp.harms@math.ethz.ch}
\thanks{
MB was supported by FWF Project P24625}
\date{\today}

\begin{abstract}
We study metrics on shape space of immersions that have a particularly simple horizontal bundle. More specifically, we consider reparametrization invariant Sobolev metrics $G$ on the space $\Imm(M,N)$ of immersions of a compact manifold $M$ in a Riemannian manifold $(N,\g)$. The tangent space $T_f\Imm(M,N)$ at each immersion $f$ has two natural splittings: one into components that are tangential/normal to the surface $f$ (with respect to $\g$) and another one into vertical/horizontal components (with respect to the projection onto the shape space $B_i(M,N)=\Imm(M,N)/\Diff(M)$ of unparametrized immersions and with respect to the metric $G$). The first splitting can be easily calculated numerically, while the second splitting is important because it mirrors the geometry of shape space and geodesics thereon. Motivated by facilitating the numerical calculation of geodesics on shape space, we characterize all metrics $G$ such that the two splittings coincide. In the special case of planar 
curves, we show that the regularity of curves in the metric completion can be controlled by choosing a strong enough metric within this class. We demonstrate in several examples that our approach allows us to efficiently calculate numerical solutions of the boundary value problem for geodesics on shape space.
\end{abstract}

\maketitle

\section{Introduction}

Nowadays, the study of the geometry of the space of all immersions of a certain type---including plane curves, space curves and surfaces---is an active area of research with applications in computational anatomy, shape comparison and image analysis, see e.g. \cite{Glaunes2008,Jermyn2011,Samir2012,Jermyn2012,Kurtek2010,Mennucci2007,Mumford2006,Scherzer2009,Bauer2014b}. This space, denoted by $\Imm(M,N)$, consists of smooth immersions from a compact $m$-dimensional manifold $M$ into a $n$-dimensional Riemannian manifold $(N,\g)$ of bounded geometry. The important special case of planar curves corresponds to $M=S^1$ and $N=\R^2$. We will always assume that $n\geq m$, in which case $\Imm(M,N)$ is a smooth Fr\'echet manifold \cite{Kriegl1990,Hamilton1982}; otherwise it is empty. 
 
The notion of shape space  in this article is that of unparametrized immersions. This space can be identified with the quotient space
\begin{equation*}
B_i(M,N)=\Imm(M,N)/\Diff(M)\,.
\end{equation*}
Here, $\Diff(M)$ denotes the Lie group of all smooth diffeomorphisms on $M$, which acts smoothly on $\Imm(M,N)$ via composition from the right:
\begin{equation*}
\Imm(M,N)\times \Diff(M)\to \Imm(M,N),\qquad (f,\ph)\mapsto f\circ\ph\,.
\end{equation*}
The quotient space $B_i(M,N)$ is not a manifold, but only an orbifold with isolated singularities (see \cite{Michor1991} for more information). 

Given a reparametrization invariant metric $G$ on $\Imm(M,N)$, we can (under certain conditions) induce a unique Riemannian metric on the quotient space $B_i(M,N)$ such that the projection
\begin{equation}\label{equ:projection}
    \pi: \Imm(M,N) \to B_i(M,N):= \Imm(M,N)/\Diff(M)
\end{equation}
is a Riemannian submersion. A detailed description of this construction is provided in \cite[Section~4]{Bauer2011b}. For many metrics, $T\pi$ induces a splitting of the tangent bundle $T\Imm(M,N)$ into a vertical bundle, which is defined as the kernel of $T\pi$, and a horizontal bundle, defined as the $G$-orthogonal complement of the vertical bundle:
\begin{equation}\label{equ:splitting1}
    T\Imm(M,N) = \on{ker}T\pi \oplus (\on{ker}T\pi)^{\bot,G} 
    =: \Ver \oplus \Hor\, .
\end{equation}
If one can lift any curve in $B_i(M,N)$ to a horizontal curve in $\Imm(M,N)$, then there is a one-to-one correspondence between geodesics on shape space $B_i(M,N)$ and horizontal geodesics on $\Imm(M,N)$ \cite[Section~4.8]{Bauer2011b}. This constitutes an effective way to compute geodesics on shape space provided that the horizontal bundle is not too complicated. A particularly favorable case is when the splitting in equation \eqref{equ:splitting1} coincides with the natural splitting into components that are tangential and normal to the immersed surface with respect to the metric $\g$:
\begin{equation}\label{equ:splitting2}
    T\Imm(M,N) = \Tan \oplus \Nor\, .
\end{equation}
If $\pi_N:TN\to N$ denotes the projection of a tangent vector onto its foot point, the above bundles are given by
\begin{equation*}\begin{split}
    T_f\Imm(M,N)&=\{h\in C^\infty(M,TN):\pi_N\circ h=f\}\,, 
    \\
    \Tan_f &=\{Tf\circ X \mid X \in \X(M)\}\,,
    \\
    \Nor_f &= \{ h \in T_f\Imm(M,N) \mid \forall k \in \Tan_f: \g(h,k)=0 \}\,.
\end{split}\end{equation*}

One specific example of a metric for which the splittings in equations \eqref{equ:splitting1} and \eqref{equ:splitting2} coincide is the $L^2$ metric, which is also the simplest and in a way most natural metric on $\Imm(M,N)$. It is defined as 
\begin{equation*}
G^{L^2}_f(h,k):=\int_M \g(h,k)\vol(g)\,.
\end{equation*}
Here, $h,k \in T_f \Imm(M,N)$ are tangent vectors at $f \in \Imm(M,N)$, $g=f^*\g$ denotes the pullback of the metric $\g$ along the immersion $f$ and  $\vol(g)$ denotes the associated volume form on $M$. 

Unfortunately, the $L^2$-metric is unsuited for many applications because the induced geodesic distance on $B_i(M,N)$ and $\Imm(M,N)$ vanishes. Indeed, any two immersions $f_0, f_1$ can be connected by paths of arbitrarily short $G^{L^2}$-length \cite{Michor2005, Bauer2012c}. The discovery of the degeneracy of the $L^2$-metric started a quest for stronger and more meaningful metrics. One particular class of metrics that has been introduced are almost local metrics \cite{Michor2007,Bauer2012a,Bauer2012b,Shah2008,YezziMennucci2005}, which are defined as
\begin{equation*}
G^{\Phi}_f(h,k):=\int_M \Phi(f)\g(h,k)\vol(g)\,,
\end{equation*}
where $\Phi:\Imm(M,N)\to \mathbb R_{>0}$ is a smooth, reparametrization invariant function. Almost local metrics enjoy the same simple splitting into horizontal and vertical subbundles as the $L^2$-metric, but overcome its degeneracy to some extent. On the positive side, they induce non-vanishing geodesic distance on $B_i(M,N)$. However, on the negative side, 
the geodesic distance vanishes along $\Diff(M)$-fibers. Moreover, well-posedness of the geodesic equation, existence of conjugated points and boundedness of the curvature tensor remain unknown for these metrics.

Another approach to strengthen the $L^2$-metric are Sobolev metrics, which have been studied in a variety of different articles, including \cite{Younes1998,Michor2007,Bauer2011b,Sundaramoorthi2011}. They are defined as
\begin{equation}\label{equ:metric}
G_f^L(h,k)=\int_M \g(L_fh,k)\vol(g)\,,
\end{equation}
where for each $f\in \Imm(M,N)$, the  operator $L_f$ is an elliptic pseudo-differential operator of order $2l$, which is symmetric and positive with respect to the $L^2$-metric. Almost local metrics are included in this definition. The standard Sobolev metric of order $n$ is
\begin{equation*}
G^{H^l}_f(h,k)=\int_M \g\left(\sum_{i=0}^l \Delta^i h,k\right)\vol(g)\,,
\end{equation*}
where the Laplacian depends on the immersion $f$ via the pullback metric $g$. 

Under certain assumptions on $L$, well-posedness of the corresponding geodesic equation has been shown \cite{Michor2007,Bauer2011b,Bauer2012d}. From a computational point of view \cite{Bauer2011a}, metrics involving a pseudo-differential operator have the severe disadvantage that the splitting into horizontal and vertical parts is rather complicated: a tangent vector $h$ is horizontal if and only if $L_fh$ is $\g$-normal to the surface $f$. Therefore, calculating the horizontal component of $h$ involves inverting the operator $L_f$ (or, to be precise, the restriction and projection of the operator $L_f$ to the space $\Tan$, cf. Lemma \ref{lem:splitting}).

If, on the other hand, the splittings \eqref{equ:splitting1} and \eqref{equ:splitting2} coincide, then projecting $h$ onto its horizontal part is easy: one simple takes the normal component of $h$. An immediate application is that one can solve the boundary value problem for geodesics on shape space efficiently by minimizing the horizontal energy of paths in the space of immersions. 
This approach has been used in \cite{Michor2006c,Michor2007,Bauer2012a,Bauer2012b} to calculate geodesics on the spaces $B_i(S^1,\mathbb R^2)$ and $B_i(S^2,\mathbb R^3)$ under various almost local metrics. However, this approach must be expected to work well for any metric where it is easy to calculate horizontal projections, e.g., for the class of metrics presented here. 

Our main result is the classification of all Sobolev metrics such that the splitting into horizontal and vertical components coincides with the natural splitting into tangential and normal components. This classification is established for immersions of general manifolds $M$ and $N$. The special case of planar curves ($M=S^1$ and $N=\mathbb R^2$) is investigated in more detail. There, it is shown that the class of metrics where horizontality equals normality is rich enough to include metrics which dominate any given standard Sobolev metric. Using a result of \cite{Bruveris2014}, it is shown that this makes it possible to control the regularity of curves in the metric completion of shape space. This is important for applications (e.g., stochastics on shape space) where going to the metric completion is indispensable. 

Finally, we present some numerical solutions of the boundary value problem for geodesics on shape space for selected metrics satisfying the assumptions of our main theorem. Our numerical scheme strongly builds on the simplicity of the horizontal bundle asserted by this theorem.

We are thankful to Martins Bruveris for helpful remarks. 

\section{The decomposition theorem}

\subsection{Assumptions}\label{subsec:assumptions}

Following \cite{Bauer2011b}, we now rigorously define the class of metrics that we study in this article. We assume that $L$ is a smooth bundle automorphism of $T\Imm(M,N)$ such that at every $f \in \Imm(M,N)$, the operator 
\begin{equation*}
L_f:T_f\Imm(M,N) \to T_f\Imm(M,N)
\end{equation*}
is a pseudo differential operator of order $2l$, which is symmetric and positive with respect to the $L^2$-metric on $\Imm(M,N)$. Moreover, we assume that $L$ is invariant under the action of the reparametrization group $\Diff(M)$ acting on $\Imm(M,N)$, i.e. 
\begin{equation*}
L_{f\circ \ph}(h\circ \ph)=L_f(h)\circ \ph \qquad \text{for all } \ph \in \Diff(M)\,.
\end{equation*}
These assumptions will remain in place throughout this work. Their immediate use is as follows: being symmetric and positive, $L$ induces a Sobolev type metric on the manifold of immersions through equation \eqref{equ:metric}. The $\Diff(M)$-invariance of $L$ implies the $\Diff(M)$-invariance of the metric $G^L$. Assuming that the decomposition in
horizontal and vertical bundles exists, there is a unique metric on $B_i(M,N)$, such that the projection \eqref{equ:projection} is a Riemannian submersion, see \cite[Thm. 4.7]{Bauer2011b}. Then  the resulting geometry of shape space is mirrored by the ``horizontal geometry'' on the manifold of immersions. 

\subsection{Splitting into horizontal and vertical subbundles}

It follows from the definitions that the horizontal and vertical bundles are given by
\begin{equation*}\begin{split}
\Ver_f&:=\ker(T\pi)=\Tan_f\,,
\\ 
\Hor_f&:=(\Ver_f)^{\bot,G}=\big\{h\in T_f\Imm(M,N): L_fh \in \Nor_f\}\,.
\end{split}\end{equation*}
By the positivity of the metric $G^L$, the sum $\Ver_f\oplus \Hor_f$ is always direct, but it might not span the entire tangent space $T_f\Imm(M,N)$. In other words, in our infinite dimensional setting, a complement to the vertical bundle might only exist in some completion of the space. In this section, we present sufficient conditions for the existence of a splitting of the tangent space into horizontal and vertical subbundles. 

In contrast to above, the splitting into tangential and normal parts given in equation \eqref{equ:splitting2} always exists. Let $h^\Tan,h^\Nor$ be the components of a tangent vector $h$ corresponding to this splitting and let
\begin{equation*}\begin{split}
L_f^\Tan: \Tan_f \to \Tan_f\,, \quad h \mapsto (L_f h)^\Tan\,,\\
L_f^\Nor: \Nor_f \to \Nor_f\,, \quad h \mapsto (L_f h)^\Nor
\end{split}\end{equation*}
be projections and restrictions of $L$ to the corresponding subbundles. These operators allow us to formulate a sufficient condition for the existence of a horizontal/vertical splitting.

\begin{lem}\label{lem:splitting}
If the operator $L_f^\Tan$ is invertible, then a splitting \eqref{equ:splitting1} into horizontal and vertical parts exists and the corresponding projections are 
\begin{align}
h^{\Ver} &= (L_f^\Tan)\i\big((L_fh)^\Tan\big)\,, &
h^{\Hor} &= h-h^{\Ver} \,.
\end{align} 
\end{lem}

\begin{proof}
It is straight-forward to check that $h^\Ver \in \Ver_f, h^\Nor \in \Nor_f$ and that these vectors add up to $h$, c.f. \cite{Bauer2011b}.
\end{proof}

The condition of Lemma~\ref{lem:splitting} is satisfied in many cases, as the following Lemma shows.

\begin{lem}[Sect. 6.8 in \cite{Bauer2011b}]\label{lem:elliptic}
If the operator $L_f$ is elliptic, then all of the operators $L_f, L_f^\Tan, L_f^\Nor$ are elliptic and invertible. 
\end{lem}

\begin{proof}
We show the statement for the operator $L_f^\Tan$. Let $\sigma^{L_f}$ and $\sigma^{L^\Tan_f}$ be the principal symbols of $L_f$ and $L^\Tan_f$. Then, for any $x \in M$ and $\xi \in T^*M$, one has
\begin{align*}
\sigma^{L_f}(\xi):T_{f(x)}N&\to T_{f(x)}N\,, & 
\sigma^{L^\Tan_f}(\xi):Tf.T_xM&\to Tf.T_xM\,.
\end{align*}
For any $\xi \neq 0,$ the mapping $\sigma^{L_f}(\xi)$ is symmetric and positive definite with respect to $\g$ because $L_f$ is elliptic. The relation 
\begin{equation*}
\forall h\in Tf.T_xM : \sigma^{L^\Tan_f}(\xi)h=\big(\sigma^{L_f}(\xi)h\big)^\Tan\,.
\end{equation*}
shows that $\sigma^{L_f^\Tan}(\xi)$ is also symmetric and positive definite. Therefore, $L_f^\Tan$ is elliptic. Moreover, it inherits symmetry and positivity with respect to the scalar product $G_f^{L^2}$ from $L_f$:
\begin{equation*}\begin{split}
\forall h,k \in \Tan_f: 
\int_M \g(L_f^\Tan h, k) \vol(g) = \int \g(h,L_f^\Tan k) \vol(g)\,, 
\\
\forall h \in \Tan_f: h \neq 0 \Rightarrow \int_M \g(L_f^\Tan h, h) \vol(g) >0\,.
\end{split}\end{equation*}
For any $j \geq 0$, let $H^j$ be the $j$-th order Sobolev completion of $\Tan_f$. Since $L_f^\Tan$ is  elliptic and symmetric, it is self-adjoint on $H^j$ and its index as an operator $H^{j+2p}\to H^j$ vanishes. It is injective (since positive) with vanishing index (since self-adjoint elliptic, by \cite[theorem 26.2]{Shubin1987}), hence it is bijective and thus invertible by the open mapping theorem. The inverse $(L_f^\Tan)\i$ restricts to a continuous linear mapping on the Fr\'echet space $\Tan_f$. Ellipticity, symmetry, positivity and invertibility of $L_f^\Nor$ and $L_f$ are shown in a similar way. 
\end{proof}

\subsection{On when horizontality equals normality}

The vertical and tangential bundles are always the same by definition (c.f. Lemma \ref{lem:splitting}). Therefore, horizontality equals normality if and only if the horizontal/vertical splitting coincides with the normal/tangential splitting, which is the property we are interested in. 

\begin{thm}[Decomposition theorem]\label{thm:decomposition}
The following two statements are equivalent:
\begin{itemize}
\item[(a)] The splitting \eqref{equ:splitting1} exists and coincides with  \eqref{equ:splitting2}.
\item[(b)] The operator $L$ has a decomposition $L=L^\Nor\oplus L^\Tan$ into operators $L^\Nor:\Nor\to\Nor$ and $L^\Tan:\Tan\to\Tan$.
\end{itemize}
\end{thm}

\begin{proof}
Assume (a), i.e., that the horizontal and normal bundles coincide. Then, by definition, $L$ maps normal into normal vectors. To see that it also maps tangential into tangential vectors, take a tangential vector $h$ and test $Lh$ against arbitrary normal vectors $k$:
\begin{equation*}
\int_M \g(Lh,k)\vol(g) = \int_M \g( h,Lk) \vol(g) =0 
\end{equation*}
for all normal vectors $k$, because $L$ is symmetric and $Lk$ is normal. This shows that $Lh$ is tangential. To summarize, we have shown that the normal and tangential bundle are invariant subspaces of $L$. Therefore, (b) holds.

Conversely, assume (b). Then, by definition, every normal vector is horizontal. To see that every horizontal vector is normal, take a horizontal vector $h$ and split it into its normal and tangential components $h=h^\Nor + h^\Tan$.  By the horizontality of $h$, $Lh$ is normal, which means that $L^\Tan h^\Tan=0$. Then also
\begin{equation*}
G^L(h^\Tan,h^\Tan)=\int_M\g( L^\Tan h^\Tan,h^\Tan)\vol(g)=0\,,
\end{equation*}
which implies $h^\Tan=0$ by the non-degeneracy of the metric. Therefore, $h$ is normal.
\end{proof}

\subsection{The geodesic equation}\label{sec:geodequation}

If $L=L^\Nor\oplus L^\Tan$ decomposes as in Theorem~\ref{thm:decomposition}, then the metric induced by $G^L$ on shape space depends on $L^\Nor$, but not on $L^\Tan$. Consequently, the geodesic equation on shape space and the horizontal geodesic equation on $\Imm(M,N)$ should also have this property. This can be verified directly. To this aim, we recall the formula for the horizontal geodesic equation of a $G^L$ metric, see \cite[Theorem~6.10]{Bauer2011b}: 
\begin{equation}\label{equ:horizontal_geodesic}
\left\{\begin{aligned}
& f_t \in \Hor\,, \\
& (\nabla_{\p_t}f_t)^\Hor = 
L\i\bigg(\frac12 \Adj(\nabla L)(f_t,f_t)^\Nor-\frac12 \g(Lf_t,f_t)\Tr^g(S)\\&\qquad\qquad\qquad\qquad\qquad-((\nabla_{f_t}L)f_t)^\Nor+\Tr^g(\g(\nabla f_t,Tf))Lf_t\bigg)\,.
\end{aligned}\right.
\end{equation}
For this equation to make sense, we need to assume that $L_f$ is invertible and that $\nabla L$ has an adjoint $\Adj(\nabla L)$ defined by
\begin{equation*}
\int \g(\Adj(\nabla L)(h,k),m)\vol=\int\g((\nabla_m L)h,k)\vol\,.
\end{equation*}
The existence of the adjoint is not guaranteed, but has to be proven for each specific operator $L$, which is usually done by a series of  integration by parts. For operators of the form $L=1+\Delta^p$, this was done in \cite[Section 8]{Bauer2011b}.

\begin{lem}\label{lem:horizontal_geodesic}
Let $L=L^\Nor\oplus L^\Tan$ as in Theorem \ref{thm:decomposition}. Then the horizontal geodesic equation \eqref{equ:horizontal_geodesic} does not depend on $L^\Tan$.
\end{lem}

\begin{proof}
Note that in equation \eqref{equ:horizontal_geodesic}, $L\i$ can be replaced by $(L^\Nor)\i$.
Thus, it is sufficient to show that the expression
\begin{multline}\label{equ:rhs_geodesic_equ}
\frac12 \Adj(\nabla L)(f_t,f_t)^\Nor-\frac12 \g(Lf_t,f_t)\Tr^g(S)
\\
-((\nabla_{f_t}L)f_t)^\Nor+\Tr^g(\g(\nabla f_t,Tf))Lf_t
\end{multline}
on the right hand side of the geodesic equation vanishes, if $L$ is of the form $L=0\oplus L^\Tan$, which we assume now. Let $m,h,k$ be vector fields on $\Imm(M,N)$ taking values in the normal bundle. Then 
\begin{multline*}
    G^{L^2}_f\big(\Adj(\nabla L)(h,k),m\big)
    = 
    G^{L^2}_f\big((\nabla_m L)h,k\big)
    \\=
    G^{L^2}_f\big(\nabla_m (\underbrace{Lh}_{=0})-L(\nabla_m h),k\big)
    =
    G^{L^2}_f\big(-\nabla_m h,\underbrace{Lk}_{=0}\big)
    =0\,,
\end{multline*}
which shows that 
\begin{equation*}
    \Adj(\nabla L)(h,k)^\Nor=0\,, \qquad ((\nabla_m L)h)^\Nor=0\,.
\end{equation*}
Since these expressions are tensorial in $m,h,k$ and $f_t$ is normal, one obtains 
\begin{equation*}
    \Adj(\nabla L)(f_t,f_t)^\Nor=0\,, \qquad ((\nabla_{f_t} L)f_t)^\Nor=0\,.
\end{equation*}
Since also $Lf_t=0$, all terms in \eqref{equ:rhs_geodesic_equ} vanish.
\end{proof}

\begin{remark}
The above Lemma shows that the horizontal geodesic equation \eqref{equ:horizontal_geodesic} is well-defined under the following conditions on $L^\Nor$: the operator $L_f^\Nor$ is invertible and $\nabla L^\Nor$ restricted to and projected onto the normal bundle has a normal bundle valued adjoint. Note that it is not necessary to impose any conditions on $L^\Tan$, here.
\end{remark}

\section{The manifold of planar curves}

In the previous section, we characterized all metrics such that the horizontal and normal bundles coincide. In practice, there are many additional properties that a metric should satisfy to be useful. For example, the induced geodesic distance should not vanish, one would like to be able to control the regularity of immersions in the metric completion, and the space should be geodesically complete. 

We show in this section that it is indeed possible to meet at least some of these objectives simultaneously. We do this in the important special case of planar curves, but we believe that similar results can be obtained for arbitrary immersions. However, as of now, metric and geodesic completeness remain unknown in the more general situation, even for standard Sobolev metrics. 

In the case of planar curves, the decomposition of tangent vectors $h\in T_c\Imm(S^1,\R^2)$ into tangential and normal components takes the particularly simple form
\begin{align*}
    h=\langle h,n\rangle n+\langle h,v\rangle v\,, \quad \text{where} \quad 
    v=\partial_\th c/|\partial_\th c|\,,\quad n=Jv\,.
\end{align*}
Here, $\langle\cdot,\cdot\rangle$ and $\lvert\cdot\rvert$ denote the Euclidean scalar product and norm on $\R^2$ and $J:\R^2\to\R^2$ denotes the counter-clockwise rotation by the angle $\pi/2$. This allows us to derive a very simple sufficient condition for the equality of the horizontal and normal bundle.

\begin{lem}\label{lem:sufficient}
Any metric of the form
\begin{equation}\label{equ:metric_ab}
G_c(an+bv,an+bv)
=\int_{S^1} \sum_{i=0}^l 
\left((A_i)_c (D_s^i a)^2 + (B_i)_c (D_s^i b)^2 \right) ds
\end{equation}
with coefficients
\begin{equation*}
A_i, B_i: \Imm(S^1,\R^2) \to C^\infty(S^1,\R)
\end{equation*}
has the property that the horizontal and normal bundles coincide.
\end{lem}

\begin{proof}
Using  integration by parts, the metric in \eqref{equ:metric_ab} can be rewritten in the form \eqref{equ:metric} involving an operator $L$, which is given by
\begin{equation*}
L_c(an+bv)
=\sum_{i=0}^l 
(-1)^i D_s^i\big((A_i)_c (D_s^i a)\big)n
+\sum_{i=0}^l 
(-1)^i D_s^i\big((B_i)_c (D_s^i b)\big)v\,.
\end{equation*}
Clearly, $\Tan_c$ and $\Nor_c$ are invariant subspaces of $L_c$. The result follows from Theorem~\ref{thm:decomposition}.
\end{proof}

Using Lemma~\ref{lem:sufficient}, we can easily construct a metric which is stronger than the Sobolev metric of order one and where horizontality equals normality.

\begin{lem}[Metrics of order one]\label{lem:stronger_than_h1}
The metric 
\begin{equation}\label{equ:metric_order_1}
G_c(an+bv,an+bv)=\int_{S^1} \left((1+2\kappa^2)(a^2+b^2)
+2(D_s a)^2+2(D_s b)^2 \right) ds
\end{equation}
dominates the standard Sobolev $H^1$ metric
\begin{equation*}
    G_c^{H^1}(h,h)= \int_{S^1} \big(|h|^2+|D_s h|^2\big)ds
\end{equation*}
and has the property that the horizontal and normal bundles coincide.
\end{lem}

\begin{proof}
The metric in \eqref{equ:metric_order_1} satisfies the condition of Lemma \ref{lem:sufficient}, which implies that the horizontal and normal bundles coincide. It remains to show that the metric dominates the $H^1$-metric. To this aim, let $h=an+bv$ with $a=\langle h,n\rangle$ and $b=\langle h,v\rangle$. Then
\begin{equation*}
D_sh = (D_s a+b\kappa)n+(D_s b-a\kappa)v\,.
\end{equation*}
Therefore, 
\begin{equation*}
    G_c^{H^1}(h,h)
    =\int_{S_1}\big( a^2+b^2+(D_s a+b\kappa)^2
    +(D_s b-a\kappa)^2\big)ds\,.
\end{equation*}
By the arithmetic-geometric inequality, the estimates
\begin{equation*}
    (D_s a+b\kappa)^2\leq 2\big((D_s a)^2+b^2\kappa^2\big)\,,
    \quad
    (D_s b-a\kappa)^2\leq 2\big((D_s b)^2+a^2\kappa^2\big)
\end{equation*}
hold. Regrouping terms, one obtains that
\begin{equation*}
    G_c^{H^1}(h,h)
    \leq 
    \int_{S^1} \left((1+2\kappa^2)(a^2+b^2)
    +2(D_s a)^2+2(D_s b)^2 \right) ds\,.\qedhere
\end{equation*}
\end{proof}

\begin{cor}
The metric $G$ from Lemma \ref{lem:stronger_than_h1} induces non-vanishing geodesic distance on both $B_i(S^1,\R^2)$ and $\Imm(M,N)$, i.e., the infimum of the $G$-lengths of paths connecting two non-identical planar curves is strictly greater than zero. Furthermore, all curves in the metric completion of the space $(B_i(S^1,{\mathbb R}^2), G)$ are Lipschitz continous.
\end{cor}

\begin{proof}
These results follow directly from the corresponding results for the $H^1$ or $G^A$ metric. See \cite{Bauer2011b,Bauer2014b} for the positivity result and 
\cite[Thm. 3.11]{Michor2006c} or \cite[Thms. 26 and 27]{Mennucci2008} for the metric completion. 
\end{proof}

\begin{remark}
Unfortunately, the well-posedness result in \cite[Section 6.6]{Bauer2011b} cannot be applied to the metric in Lemma \ref{lem:stronger_than_h1}. To see this, one has to rewrite the metric in the form \eqref{equ:metric} involving an operator $L$. In the present case, 
\begin{equation*}
Lh=(1+2\kappa^2)h-2\left(D_s^2 \langle h,n\rangle\right)n
-2\left(D_s^2 \langle h,v\rangle\right) v.
\end{equation*}
The problem is that the expression on the right hand side contains third derivatives of $c$ while $L$ is only of second order in $h$. Similar problems arise with higher order metrics and metrics on spaces of higher dimensional immersions.
\end{remark}

\begin{thm}\label{strongerthensob}
For any $l \geq 0$, there is a metric which dominates the Sobolev $H^l$ metric 
\begin{equation*}
    G^{H^l}(h,h):= \int_{S^1}\sum_{i=0}^l |D^i_s h|^2 ds
\end{equation*}
and which has the property that the horizontal and normal bundles coincide.
\end{thm}

\begin{proof}
The statement for $l=0$ is trivial. For $l=1$, it has been shown in Lemma \ref{lem:stronger_than_h1}. For $l=2$, it remains to bound the highest order term. An application of Jensen's inequality yields the estimate
\begin{multline*}
|D_s^2(an+bv)|^2\leq 4(D^2_s a)^2 + 16(D_s a)^2\ka^2 +4a^2(D_s\ka)^2+4a^2\ka^4
\\
+4(D^2_s b)^2 + 16(D_s b)^2\ka^2 +4b^2(D_s\ka)^2+4b^2\ka^4\,.
\end{multline*}
Letting $h=an+bv$ and taking in account also the zero and first order terms, c.f. Lemma \ref{lem:stronger_than_h1}, one obtains
\begin{multline*}
G^{H^2}(h,h)\leq
\int_{S^1}\Big(
4(D^2_s a)^2 
+ \left(16\ka^2+2\right)(D_s a)^2 +\left(4(D_s\ka)^2+4\ka^4+2\kappa^2+1\right)a^2
\\
+4(D^2_s b)^2 
+ \left(16\ka^2+2\right) (D_s b)^2
+\left(4(D_s\ka)^2+4\ka^4+2\kappa^2+1\right)b^2\Big)ds\,.
\end{multline*}
The expression on the right-hand side defines a metric, which is of the form required in Lemma \ref{lem:sufficient}. Therefore, the horizontal bundle equals the normal bundle for this metric. By construction, it dominates the $H^2$ metric. For higher order metrics, the proof is similar.
\end{proof}

Since we have shown that it is possible to construct arbitrarily strong metrics for which horizontality equals normality, we are able to control the regularity of curves in the metric completion. 

\begin{cor}
Let $G$ be a metric which is stronger than the $H^l$ metric. Then the metric completion of the space $(\Imm(S^1,{\mathbb R}^2), G)$ is contained in the set $\on{Imm}^l(S^1,\R^2)$ of all $H^l$ immersions.
\end{cor}
\begin{proof}
Every Cauchy sequence with respect to the metric $G$ is also a Cauchy sequence with respect to the weaker metric $H^l$. By Theorem \ref{thm:completeness}, the sequence has a limit in the space $\Imm^l(S^1,\mathbb R^2)$. Therefore, the completion of $(\Imm(S^1,{\mathbb R}^2), G)$ can be seen as a subset of $\Imm^l(S^1,\mathbb R^2)$. 
\end{proof}
See also \cite{Mennucci2008} for the metric completion of a length-weighted Sobolev $H^2$-metric.

\section{Numerical results}

In this section, we demonstrate that the boundary value problem for geodesics on shape space can be solved efficiently for metrics satisfying the assumptions of Theorem \ref{thm:decomposition}. In particular, we are able to treat higher order metrics without much effort.

\subsection{The boundary value problem for geodesics}

We present several examples of (discrete) solutions of the boundary value problem for geodesics of planar curves, i.e., we are searching for a geodesic $c(t,\cdot)$ connecting two given boundary curves $c_0,c_1\in B_i(S^1,\R^2)$. We tackle this problem by minimizing the horizontal path energy
\begin{equation}\label{equ:horizontal_engergy}
    E^{\on{hor}}(c) = \int_0^1 G(c_t^{\on{hor}},c_t^{\on{hor}}) dt
\end{equation}
over the set of paths in $\Imm(S^1,\R^2)$ with fixed endpoints. Here, $c_t$ denotes the time derivative $\partial_t c$ of the path. The horizontal component $c_t^{\on{hor}}$ is easy to calculate whenever the metric satisfies the conditions of Theorem \ref{thm:decomposition}.

Note that the energy functional \eqref{equ:horizontal_engergy} does not depend on the parametrization of the curve at each instant of time. So we are free to choose a suitable parametrization. We do this by adding to the energy functional a term penalizing  irregular parametrizations.  So instead of minimizing the horizontal path energy, we minimize the sum of horizontal path energy and a penalty term, which measures the deviation from constant speed parametrization.

\subsection{The metrics of interest}\label{sec:metrics_of_interest}
In our numerical experiments, we consider four different metrics of order zero, one and two, to be described below. All of them satisfy the conditions of Theorem \ref{thm:decomposition}. For horizontal tangent vectors $h=an$ they are given by:
\begin{itemize}
 \item Metric 1 ($G^A$-metric with $A=2$):
 \begin{equation*}
  G_c^1(an,an)=\int_{S^1} (1+2\kappa^2)a^2  ds
 \end{equation*}
 \item Metric 2 ($H^1$-type metric):
 \begin{equation*}
 G_c^2(an,an)=\int_{S^1} \left((1+2 \kappa^2)(a^2)
+2(D_s a)^2 \right)ds
 \end{equation*}
 \item Metric 3: 
 \begin{equation*}
   G_c^2(an,an)=
 \int_{S^1}\left(4(D_s\ka)^2+4\ka^4+2\kappa^2+1\right)a^2 ds
 \end{equation*}
 \item Metric 4 ($H^2$-type metric):
  \begin{align*}
   G_c^4(an,an)=&
 \int_{S^1}\Big(
4(D^2_s a)^2 
+ \left(16\ka^2+2\right)(D_s a)^2 \\&\qquad+\left(4(D_s\ka)^2+4\ka^4+2\kappa^2+1\right)a^2\Big) ds
 \end{align*}
\end{itemize}
The motivation for Metrics 2 and 4 lies in Theorem \ref{strongerthensob}: Metric 2 dominates the $H^1$-metric and Metric 4 dominates the $H^2$-metric. 

Metric 1 (the $G^A$-metric) has been introduced by Michor and Mumford \cite{Michor2006c} to overcome the degeneracy of the $L^2$-metric. This metric involves all terms of the $H^1$-metric which 
do not differentiate the tangent vector, but only the foot point, see \cite[Sect. 3.2]{Michor2006c}. Metric 3 has a similar property: it consists of all terms of the $H^2$-metric which do not differentiate the tangent vector.

\subsection{Discretization of the energy functional}\label{sec:discretization}

To obtain approximate solutions of the infinite dimensional problem, we discretize the space of curves and the energy functional. We follow the same maxims as \cite{Michor2006c,Michor2007,Bauer2012a,Bauer2012b}. 

Curves are replaced by polygons with a fixed number of vertices. According to the principles of discrete differential geometry, geometric quantities are defined on their natural domain (edges, vertices, triangles, etc.) and are assigned weights corresponding to their domain. When necessary, domains are split into sub-domains and the weights are adjusted accordingly. This ensures that the integrals in the definition of the horizontal energy functional are evaluated properly. 

For example, the discrete curvature $\kappa$ is defined via the turning angle between two adjacent edges. Its domain consists of the two adjacent half-edges and the corresponding weight is the total length of the half-edges. As a further example, the normal vector is defined on the edges of the curve, whereas $c_t$ is defined on vertices. Thus, to define the function $a=\langle c_t,n\rangle$, the domains have to be split: the function $a$ is defined on triangles spanned by a vector $c_t$ (obtained by either forward or backward differences) and a neighboring edge of the curve (again obtained by either forward or backward differences). Note that around each vertex, there are four such triangles. To calculate spatial derivatives of $a$, one needs to define the ``distance'' between adjacent triangles. If the triangles share a common edge, then the distance is defined as half of the edge length. If they share a common vertex, it is defined as half of the weight assigned to the vertex. Now one can define discrete 
versions of $D_s a$ and $D_s^2 a$ by taking 
first and second order finite differences, respectively.

To avoid artifacts during the optimization procedure, the energy functional is written as sums of squares instead of squares of sums, whenever possible. This can be achieved by splitting domains into sub-domains as explained above. 

The details of our implementation can be seen from the AMPL source code listed in Appendix B.

\subsection{Numerical implementation}

As in \cite{Bauer2011a}, we use the nonlinear solver IPOPT (Interior Point OPTimizer \cite{Waechter2002}) in combination with AMPL (A Modeling Language for Mathematical Programming \cite{Fourer2002}). IPOPT applies a filter based line search method to compute the minimum of the discretized horizontal energy functional. In this process, it needs the gradient and the Hessian of the energy, which is automatically and symbolically calculated by AMPL. The full AMPL-code needed to reproduce our results is printed in Appendix \ref{ampl}.

\subsection{Description of the numerical experiments}

Numerics for first order metrics on the shape space of planar curves (and, in particular, the elastic metric) are well developed \cite{Younes1998,Klassen2004,Jermyn2011}. A recurrent issue with these metrics is that singularities can form along geodesics. This leads to severe problems, including a-symmetry of the numerically calculated geodesic distance function \cite{Bauer2014b}. It is a natural question if these problems can be alleviated by using higher order metrics. This is indeed the case in our numerical experiments, where we compare a first order metric (Metric 2) to a second order one (Metric 4). An example is presented in Figure \ref{fig:cat2cow} showing a geodesic connecting the silhouettes of a cat and a cow.\footnote{The shapes are taken from the database of closed binary shapes collected by the LEMS Vision Group at Brown university (\url{http://www.lems.brown.edu/~dmc}).} For the first order metric, one can clearly observe the appearance of singularities (kinks) along the geodesic path. Such 
singularities did not occur for the higher order metric (Metric 4). In the case of simpler deformations, both metrics perform well (cf. Figure \ref{fig:circle2star}).

\begin{figure*}
\begin{center}
\includegraphics[width=.48\textwidth]{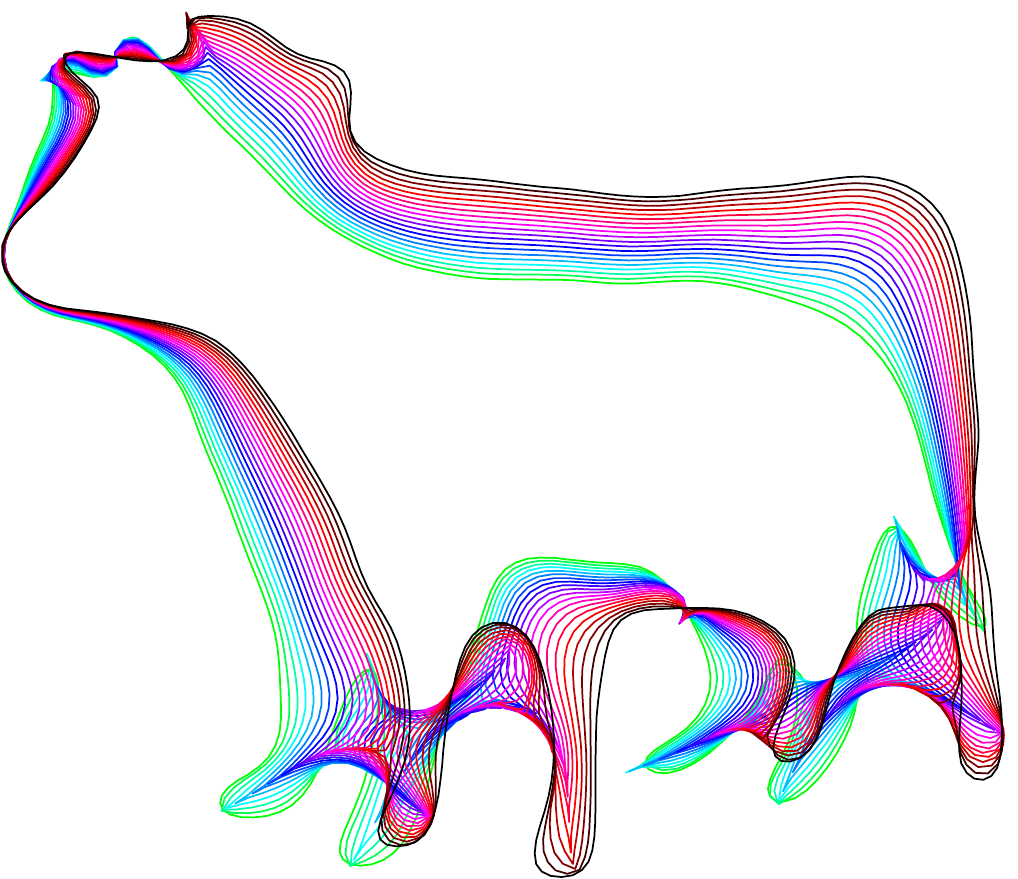} 
\includegraphics[width=.48\textwidth]{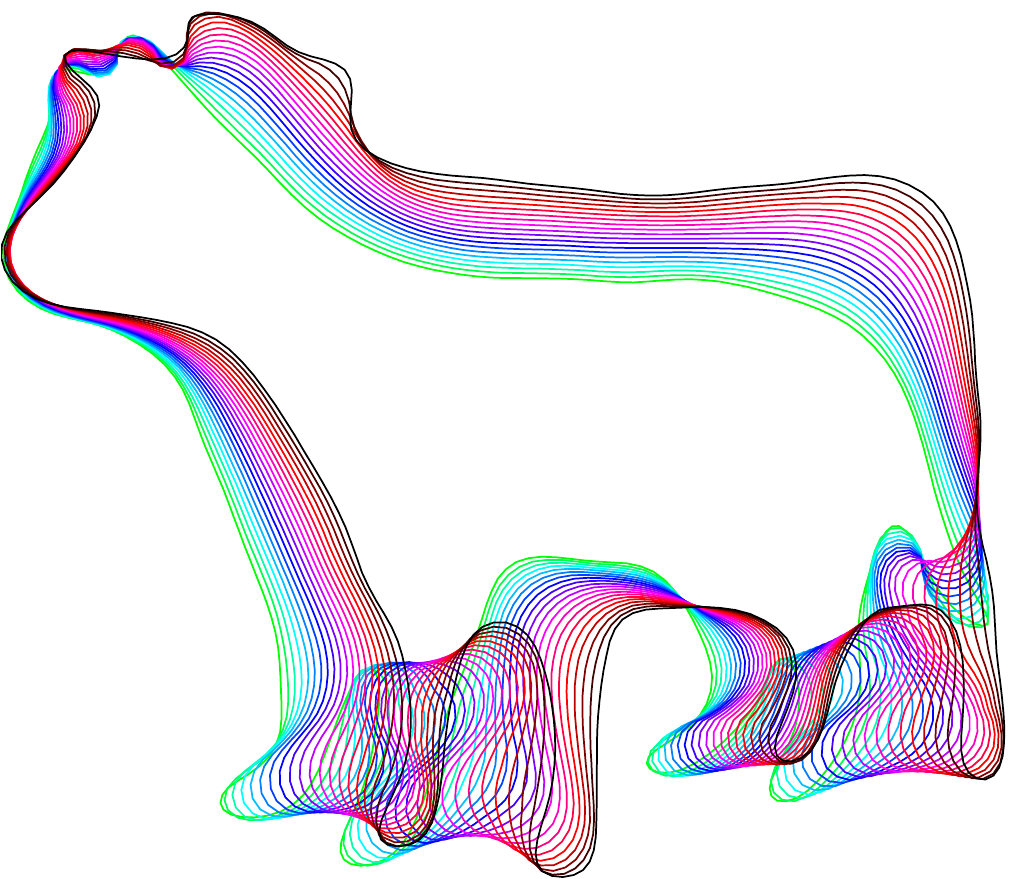}
\includegraphics[width=.96\textwidth]{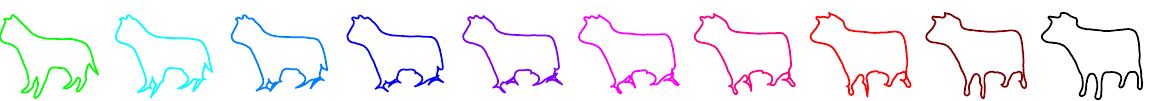}
\includegraphics[width=.96\textwidth]{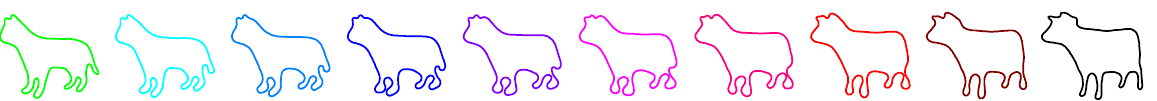}
\end{center}
\caption{Geodesic between a cat- and a cow-shaped curve. Singularities appear under Metric 2 (left and top), but not under Metric 4 (right and bottom).}
\label{fig:cat2cow}
\end{figure*}

\begin{figure*}
\begin{center} 
\includegraphics[width=.99\textwidth]{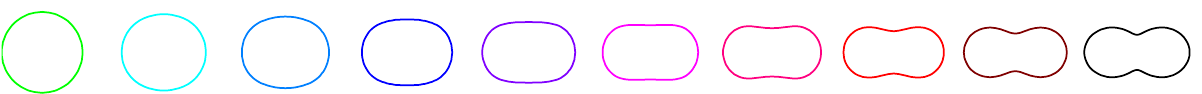} 
\includegraphics[width=.99\textwidth]{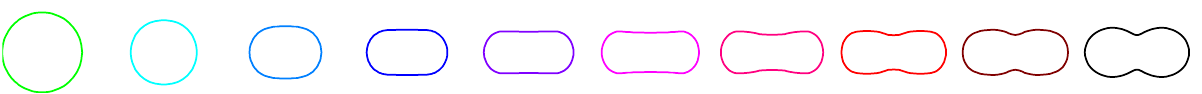} 
\includegraphics[width=.99\textwidth]{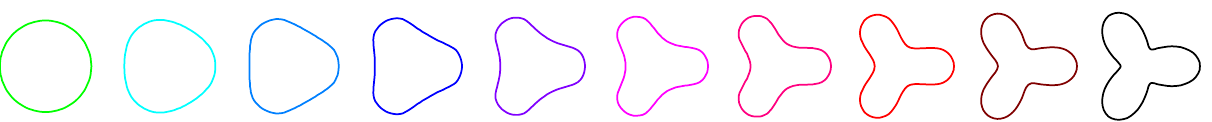} 
\includegraphics[width=.99\textwidth]{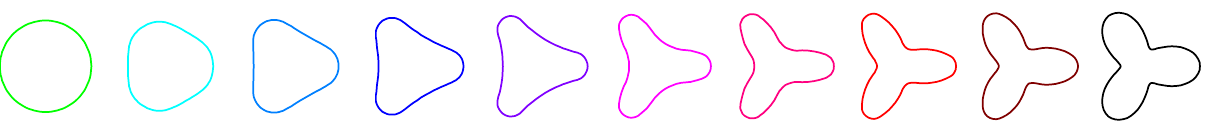} 
\includegraphics[width=.99\textwidth]{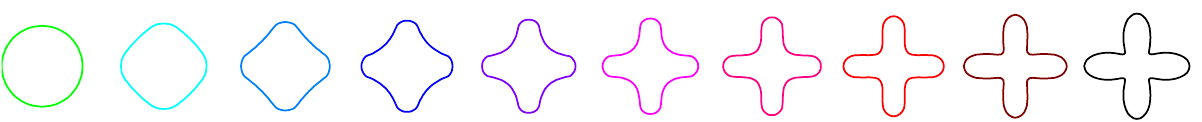} 
\includegraphics[width=.99\textwidth]{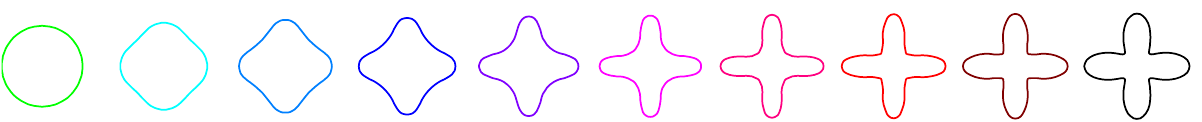} 
\includegraphics[width=.99\textwidth]{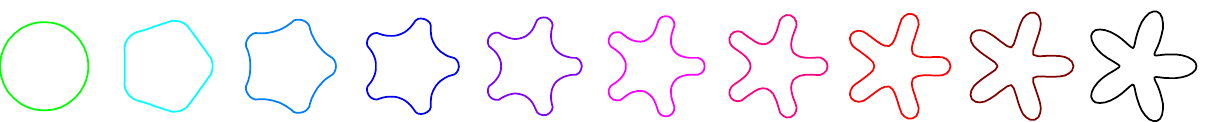} 
\includegraphics[width=.99\textwidth]{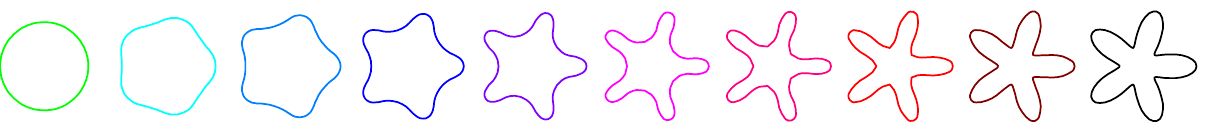} 
\end{center}
\caption{Geodesics between various simple shapes under Metric 2 (lines 1, 3, 5, and 7) and Metric 4 (lines 2, 4, 6, and 8). Both metrics perform well.}
\label{fig:circle2star}
\end{figure*}

\begin{figure*}
\begin{center}
\includegraphics[width=.49\textwidth]{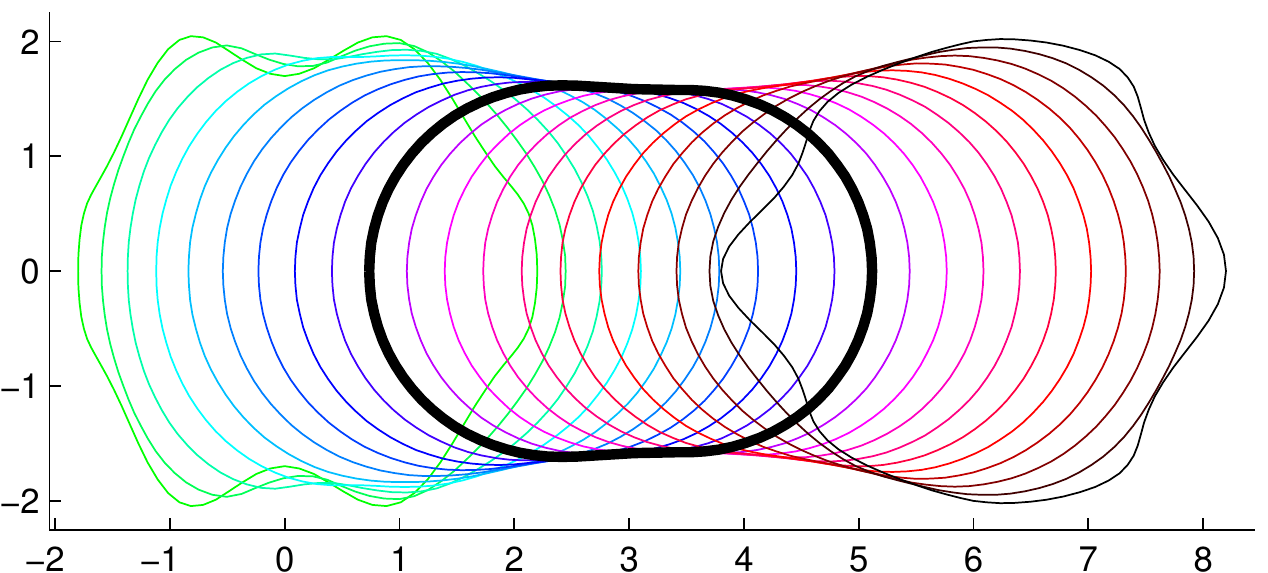} 
\includegraphics[width=.49\textwidth]{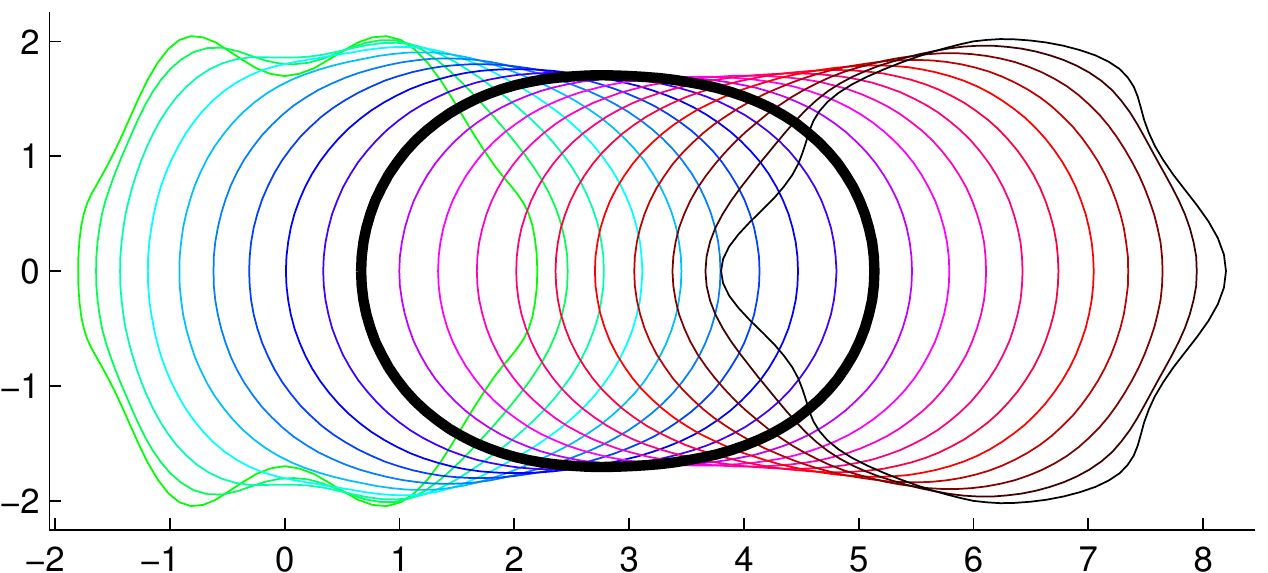} 
\includegraphics[width=.49\textwidth]{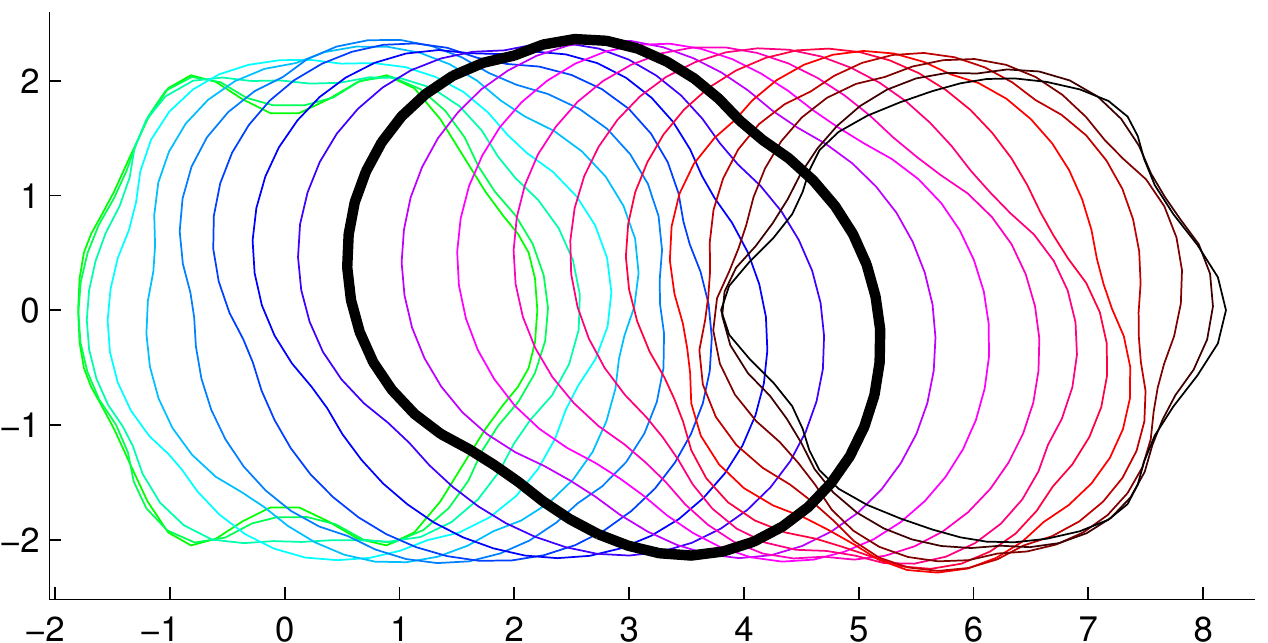} 
\includegraphics[width=.49\textwidth]{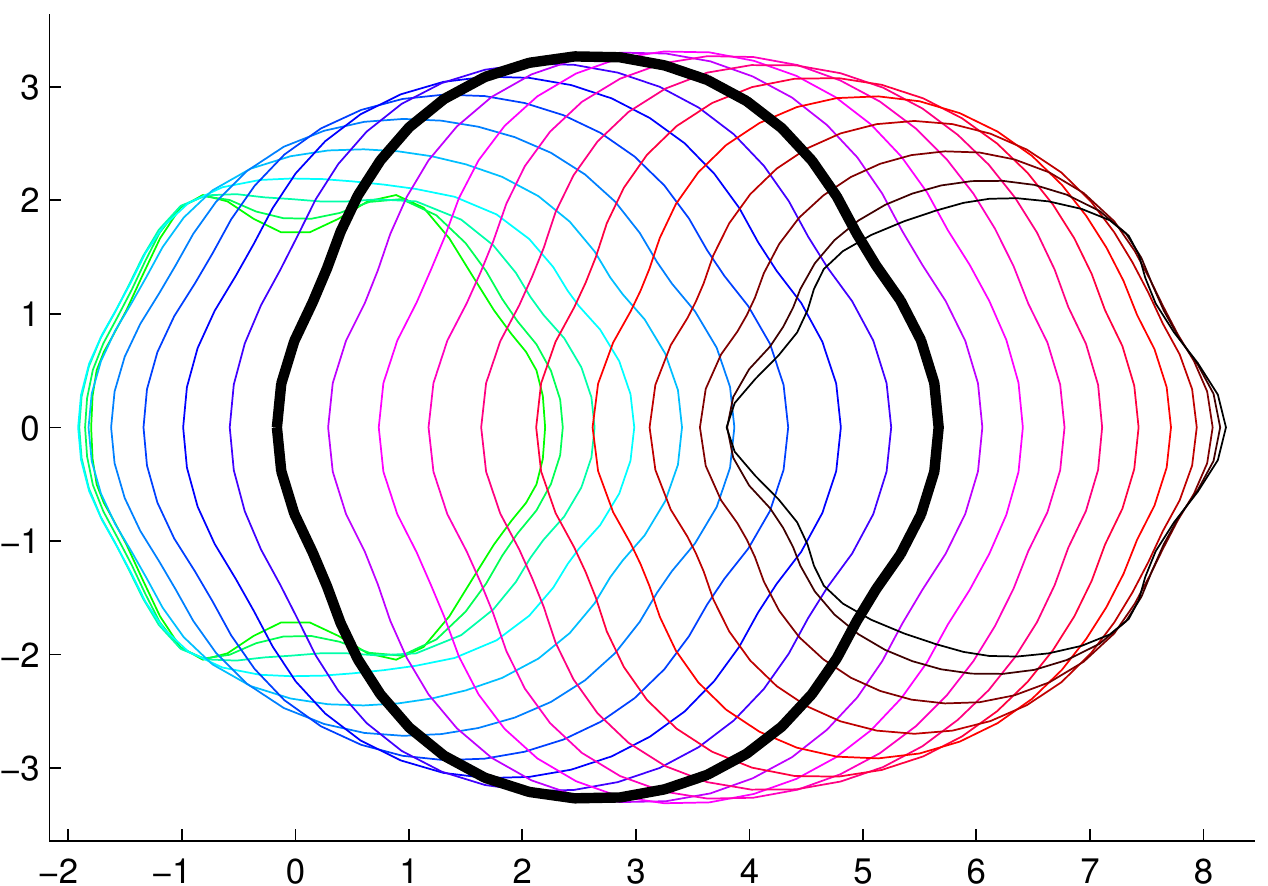} 
\end{center}
\caption{Geodesics between two distinct curves under Metrics 1--4 (top left to bottom right). High frequency components of the shapes quickly disappear under Metrics 1 and 2.}
\label{fig:move}
\end{figure*}

High-frequency variations of shapes are felt more strongly by higher order metrics. Thus, one might guess that these metrics do a better job at preserving the fine structure of a shape along a deformation affecting predominantly its coarse structure. Such an effect can be observed in Figure \ref{fig:move} showing geodesics between two distinct shapes under all four metrics.

Another observation is that geodesics between shapes lying far apart in the plane tend to pass through one and the same shape, which (depending on the metric) might resemble an ellipse or a cigar. This phenomenon, which might be connected to regions of negative curvature in shape space, has first been observed for the $G^A$-metric in \cite{Michor2006c} and for almost local metrics in \cite{Bauer2012a}. Here, we show that the same effect also occurs for second order metrics, albeit only at higher distances of the shapes (cf. Figure \ref{fig:move_long}). 

\begin{figure*}
\begin{center}
\includegraphics[width=.98\textwidth]{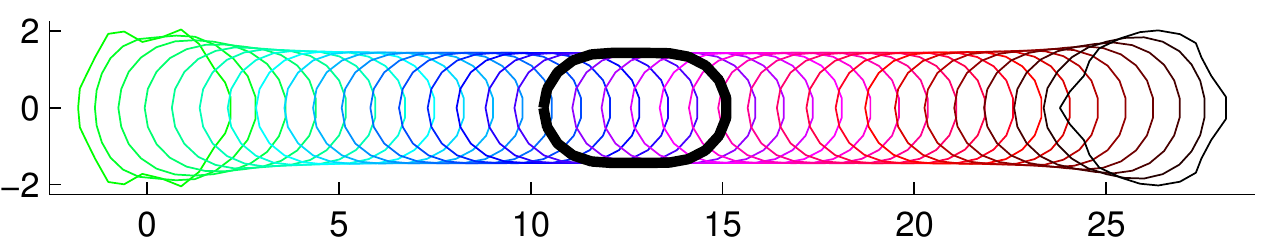} 
\includegraphics[width=.98\textwidth]{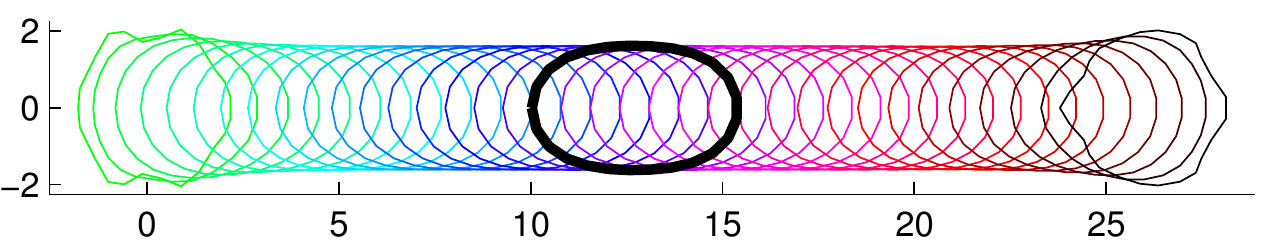} 
\includegraphics[width=.98\textwidth]{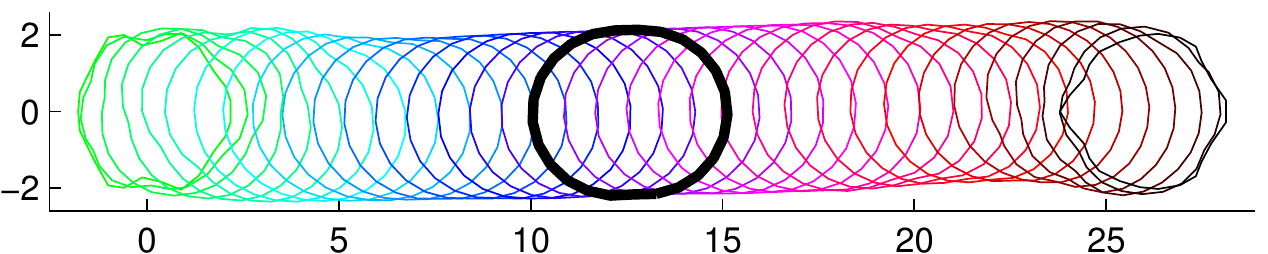} 
\includegraphics[width=.98\textwidth]{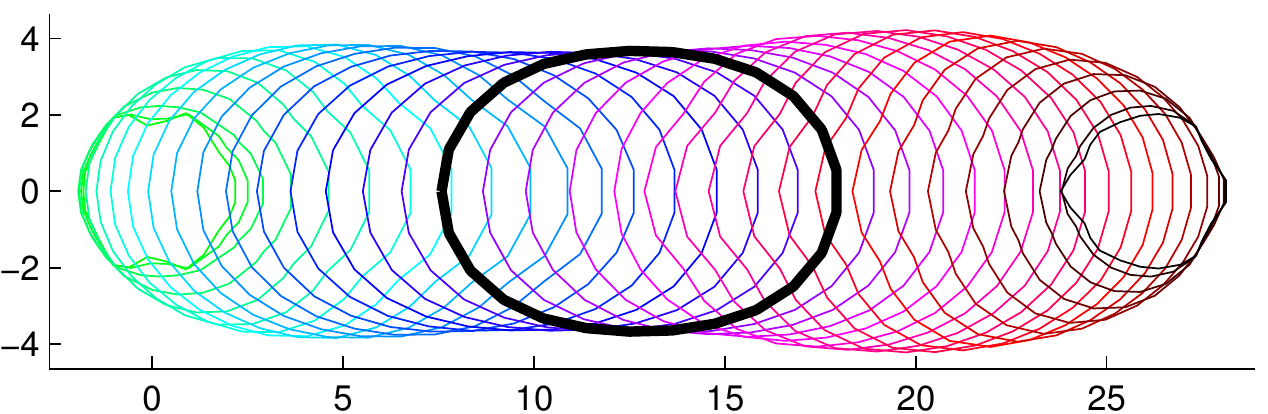} 
\end{center}
\caption{Geodesics between sufficiently distant shapes tend to go through an ellipse- or cigar-like shape (Metrics 1--4).}
\label{fig:move_long}
\end{figure*}

\section{Conclusion}

We characterized all metrics on spaces of immersions such that the horizontal/vertical splitting coincides with the normal/tangential one (Theorem \ref{thm:decomposition}). The first splitting can be easily calculated numerically, while the second splitting is important both theoretically and in applications because it mirrors the geometry of shape space and geodesics thereon. 

To show that the condition that horizontality equals normality is not too stringent, we studied the special case of planar curves in more detail. We proved in this case that, for any Sobolev metric of given order, there is a stronger metric with the property that horizontality equals normality. By a result of \cite{Bruveris2014}, this implies that the regularity of curves in the metric completion can be controlled. Therefore, our class of metrics is rich enough for applications (e.g., stochastics on shape space) where working with the metric completion is indispensable.

Finally, we demonstrated in numerical examples that the boundary value problem for geodesics on shape space can be solved efficiently for these metrics.

\appendix

\section{Completeness results for Sobolev metrics}

On the manifold of planar curves, one can define a reparametrization invariant Sobolev metric  by
\begin{equation}\label{Sobolev_metric}
    G^{H^l}_c(h,h)
    = \int_{S^1} \left\langle \sum_{i=0}^l(-1)^i D^{2i}_s h,h\right\rangle ds
    =\int_{S^1}\sum_{i=0}^l |D^i_s h|^2 ds\,.
\end{equation}
Recently, it has been shown in \cite{Bruveris2014} that the $H^l$ metric \eqref{Sobolev_metric} is geodesically complete for $l\geq 2$. Since the Hopf-Rinow theorem does not hold in  infinite dimensions \cite{Atkin1975}, this does not automatically imply metric completeness. Instead, we have the following result:

\begin{thm}\label{thm:completeness}
For $l\geq 2$, the metric completion of the space of smooth immersed curves endowed with the $H^l$ metric \eqref{Sobolev_metric} is the space of all $H^l$ immersions
\begin{equation*}
\on{Imm}^l(S^1,\R^2):=\{c\in H^l(S^1,\R^2): |c'|\neq 0\}\,.
\end{equation*}
\end{thm}

\begin{remark}
Note that $G^{H^l}$ is a strong Riemannian metric on the Hilbert manifold $\on{Imm}^l(S^1,\R^2)$.
Thus, the topology on $\on{Imm}^l(S^1,\R^2)$ induced by the geodesic distance of the metric $G^{H^l}$ is equal to the standard topology on $H^l(S^1,\R^2)$.
\end{remark}

\begin{proof}
The proof of this theorem is based on a comparison with the flat (non-reparametrization invariant) metric given by
\begin{equation*}
    G^{H^l(d\theta)}_c(h,h)= \int_{S^1}\sum_{i=0}^l |\partial^i_{\theta} h|^2 d\theta\,.
\end{equation*}
Note that in contrast to the reparametrization invariant metric, this metric does not depend on the foot point $c$. To highlight the difference between these metrics, we use the notations $H^l(ds)$ and $H^l(d\theta)$. 

According to  \cite[Lemma 5.1 and Lemma 4.10]{Bruveris2014}, the following holds for each $r>0$ and $c_0 \in \Imm(S^1,\R^2)$.
\begin{itemize}
\item[(a)] On the $H^l(ds)$-metric ball $B^{ds}_r(c_0)$, the $H^l(ds)$-metric is equivalent to the flat $H^l(d\th)$-metric, with a constant that depends only on the center $c_0$ and the radius $r$.
\item[(b)] There exists a $C=C(r,c_0)> 0$ such that $|c'|\geq C$ for all $c\in B^{ds}_r(c_0)\,.$
\end{itemize}

Let $c_k$ be a Cauchy-sequence in $\on{Imm}(S^1,\R^2)$ with respect to the $H^l(ds)$-geodesic distance. Then, for each $r >0$ there exists $N>0$ such that for each $m\geq N$ the curve $c_m$ lies in the $H^l(ds)$-metric ball $B_r(c_N)$.  By (a), the sequence $c_k$ is also a Cauchy sequence with respect to the flat metric $H^l(d\theta)$. Therefore, the sequence has a limit $c$ in the Hilbert space $H^l(S^1,\R^2)$. By (b), we conclude that $c$ is still an immersion. Thus, we have shown that the metric completion of $(\Imm(S^1,\R^2),H^l(ds))$ is contained in $(\Imm^l(S^1,\R^2),H^l(ds))$. In fact, we have equality because any curve in $\Imm^l(S^1,\R^2)$ can be approximated in the $H^l(ds)$-metric by a sequence of curves in $\Imm(S^1,\R^2)$. 
\end{proof}

\section{AMPL}\label{ampl}

In this section, we show the AMPL model file for the minimization of energy functionals of the form
\begin{multline*}
\int_0^1 \int_{S^1} \Big[(A_0+A_1 \kappa^2+A_2 \kappa^4 
+ A_3(\partial_s\kappa)^2) a^2 
\\
+(B_0+B_1\kappa^2) (\partial_s a)^2 
+ C_0 (\partial_s^2 a)^2\Big] ds dt.
\end{multline*}
This encompasses all four metrics defined in Section \ref{sec:metrics_of_interest}. The discretization is based on the principles laid out in Section \ref{sec:discretization}.

We describe the notation used in the code. The parameters $T$ and $N$ denote the number of time steps and the number of vertices, respectively. $V=\{1,\ldots,N\}$ is the set of all vertices. Time indices are denoted by $s,t$ and spatial indices by $v,w$. When both $s$ and $t$ are used simultaneously, then $s$ determines the surface whose geometry is used in the evaluation of geometric quantities like $\kappa, n$, etc., whereas $t$ determines the time increment which is used in the evaluation of $c_t$. In a similar spirit, when both indices $v$ and $w$ are used simultaneously, then $v$ determines the base vertex of a triangle and $w$ the neighboring edge. 

\lstset{basicstyle=\small\ttfamily,
  emph={param,sum,var,sum,if,else,minimize,then,sqrt,abs,acos,set},
  emphstyle=\textbf, extendedchars=true,
  backgroundcolor=\color[gray]{0.9},
  numbers=left, numberstyle=\scriptsize, stepnumber=2, numbersep=5pt,
  xleftmargin=12pt, xrightmargin=12pt,columns=flexible,
  showstringspaces=false,morecomment=[l]{\#}}

\begin{lstlisting}[caption=AMPL model file]
param T > 1 integer;
param N integer;
param Pi default 3.141592653589793;
param A0 default 1; param A1 default 2; 
param A2 default 4; param A3 default 4;
param B0 default 2; param B1 default 16;
param C0 default 4;
set V := 1..N circular;
param c0 {V,1..2}; param c1 {V,1..2};
# free vertices, i.e., those at t=2,...,T
var cmiddle {t in 2..T,v in V,i in 1..2} := 
  ((T-t+1)*c0[v,i]+(t-1)*c1[v,i])/T;
# all vertices
var c {t in 1..T+1,v in V,i in 1..2} = 
  (if t=1 then c0[v,i] else if t=T+1 then c1[v,i] 
  else cmiddle[t,v,i]);
# edges
var c_x {t in 1..T+1,v in V,i in 1..2} =
  c[t,next(v,V),i]-c[t,v,i];
# volume form defined on edges
var vol_edge {t in 1..T+1,v in V} =
  sqrt(c_x[t,v,1]^2+c_x[t,v,2]^2);
# volume form defined on vertices
var vol_vert {t in 1..T+1,v in V} =
  (vol_edge[t,prev(v,V)]+vol_edge[t,v])/2;
# volume form defined on triangles
var vol_tri {t in 1..T+1,v in V,w in {prev(v,V),v}} =
  vol_vert[t,v]+vol_edge[t,w];
# tangent vector defined on edges
var tangent {t in 1..T+1,v in V,i in 1..2} =
  c_x[t,v,i]/vol_edge[t,v];
# normal vector defined on edges
var n {t in 1..T+1,v in V,i in 1..2} = (
  if i=1 then 
    tangent[t,v,2]
  else 
    -tangent[t,v,1]);
# total length of the curve
var length {t in 1..T+1} = (
  sum {v in V} vol_edge[t,v]);
# turning angle defined on vertices
var angle {t in 1..T+1,v in V} =
  acos(tangent[t,prev(v,V),1]*tangent[t,v,1]
    +  tangent[t,prev(v,V),2]*tangent[t,v,2]);
# curvature defined on vertices
var kappa {t in 1..T+1,v in V} =
  angle[t,v]/vol_vert[t,v];
# spatial derivative of curvature defined on edges
var kappa_s {t in 1..T+1,v in V} =
  (kappa[t,next(v,V)]-kappa[t,v])/vol_edge[t,v];
# velocity defined on vertices
var c_t {t in 1..T,v in V,i in 1..2} =
    T*(c[t+1,v,i]-c[t,v,i]);
# function a defined on triangles
var a {t in 1..T,s in t..t+1,v in V,w in {prev(v,V),v}} =
  sum {i in 1..2} c_t[t,v,i]*n[s,w,i];
# spatial derivative of a defined on triangles
var a_s {t in 1..T,s in t..t+1,v in V,w in {prev(v,V),v}}=(
  if w=v then
    sum {i in 1..2} c_t[t,v,i]*(n[s,w,i]-n[s,prev(w,V),i]) 
    +
    sum {i in 1..2} (c_t[t,next(v,V),i]-c_t[t,v,i])*n[s,w,i] 
  else # in this case, w=v-1
    sum {i in 1..2} c_t[t,v,i]*(n[s,next(w,V),i]-n[s,w,i]) 
    +
    sum {i in 1..2} (c_t[t,v,i]-c_t[t,prev(v,V),i])*n[s,w,i] 
  ) / vol_tri[s,v,w];
# second spatial derivative of a defined on triangles. 
var a_ss {t in 1..T, s in t..t+1, v in V, w in {prev(v,V),v} } =
  if w=v then
    (
      (a[t,s,next(v,V),w]-a[t,s,v,w])/vol_edge[s,w]
      -
      (a[t,s,v,w]-a[t,s,v,prev(w,V)])/vol_vert[s,v]
    )/(vol_edge[s,w]+vol_vert[s,v])
  else # in this case, w=v-1
    (
      (a[t,s,v,next(w,V)]-a[t,s,v,w])/vol_vert[s,v]
      -
      (a[t,s,v,w]-a[t,s,prev(v,V),w])/vol_edge[s,w]
    ) / vol_tri[s,v,w];
# penalty term defined on vertices
var penalty =
  sum {t in 1..T+1,v in V} (vol_edge[t,v]-length[t]/N)^2;
# energy at each time step
var energy {t in 1..T} = 
  (
    sum {s in t..t+1, v in V, w in {prev(v,V),v} } 
    (
      (A0+A1*kappa[s,v]^2+A2*kappa[s,v]^4+A3*kappa_s[s,w]^2) 
      * a[t,s,v,w]^2 +
      (B0+B1*kappa[s,v]^2) * a_s[t,s,v,w]^2
      +
      C0 * a_ss[t,s,v,w]^2
    ) * vol_tri[s,v,w]
  )/8;
# total energy, summed up over all time steps
var total_energy = 
  (sum {t in 1..T} energy[t])/T;
# objective functional
minimize f: 
  total_energy+penalty;

\end{lstlisting}

\def\cprime{$'$}

\end{document}